\documentclass[12pt,psamsfonts]{article}
\usepackage{amsmath}
\usepackage{amssymb}
\usepackage{amsthm}

\newtheorem{defn0}{Definition}[section]
\newtheorem{prop0}[defn0]{Proposition}
\newtheorem{thm0}[defn0]{Theorem}
\newtheorem{lemma0}[defn0]{Lemma}
\newtheorem{corollary0}[defn0]{Corollary}
\newtheorem{example0}[defn0]{Example}
\newtheorem{remark0}[defn0]{Remark}
\newtheorem{conjecture0}[defn0]{Conjecture}

\newenvironment{theorem}{\bigskip \begin{thm0}}{\end{thm0}}

\newenvironment{corollary}{\bigskip \begin{corollary0}}{\end{corollary0}}

\def\cocoa
{\mbox{\rm C\kern-.13em o\kern-.07
em C\kern-.13em o\kern-.15em A}}

\newcommand{\m}{\mathfrak{m}}
\newcommand{\n}{\mathfrak{n}}

\title{  \bf \huge A family of local rings with rational Poincar{\'e} Series
}
\author{\large   Juan Elias
\thanks{Partially supported by  MEC-FEDER MTM2007-67493}
\and \large Giuseppe Valla \
\thanks{Partially supported by MIUR}}

\date{\today}

\begin{document}

\maketitle

\begin{abstract}
In this note we compute  the Poincare Series of almost stretched Gorenstein local rings.
 It turns out that it is rational.\end{abstract}

\bigskip
\section{Introduction.}  Let $(R,\n)$ be a regular local ring and  $k=R/\n$ its residue
field which we assume of characteristic zero.

Given an ideal $I\subseteq \n^2$, a classical problem in Commutative
Algebra is to study the Poincar{\'e} series
$$\mathbb{P}_A(z):=\sum_{i\ge 0}\dim_kTor_i^A(k,k))z^i$$ of the
local ring $(A=R/I,\m=\n/I).$  This is the generating function of
the sequence of Betti numbers of  a minimal free resolution of $k$
over $A.$

Due to the classical conjecture of Serre, the main issue is concerning the rationality of this series.
We know by the example of Anick, see \cite{A},  that this series can be non rational,
but  there  are relatively few classes of local rings for which the question has been settled.
See \cite{Av} for a detailed study of these and other relevant related problems in local algebra.

 Given a Cohen-Macaulay local ring $A=R/I,$ we say that $A$ is {\bf stretched} if there exists
 an artinian reduction $B$ of $A$ such that the square of its maximal ideal is a principal ideal.
 Instead, if the square of the maximal ideal of an artinian reduction is minimally generated by
 two elements, we say that $A$ is {\bf almost stretched}. See \cite{S1}, \cite{RV}, \cite{EV1}
 and \cite{EV2} for papers concerning   these notions.

 In \cite{S2} J. Sally computed  the Poincar{\'e} series of a stretched Cohen-Macaulay local
 rings and obtained, as a corollary,  the rationality of the series. It follows  that local
 Gorenstein rings of multiplicity at most five have rational Poincar{\'e} series.

 In this paper we compute   the Poincar{\'e} series of an almost stretched Gorenstein local
  ring, thus exhibiting  its rationality. Using this result, we can prove the rationality of
   the Poincar{\'e} series of any Gorenstein local rings of multiplicity at most seven.

 We are not developing new methods for the computation of the Betti numbers of the minimal
  $A$-free resolution of $k;$  rather we show  that the structure theorem we proved in \cite{EV1}
   for Artinian almost stretched Gorenstein local rings is very much suitable to  the computation of $Tor_i^A(k,k).$

 In the following, for  a local ring $(A,\m,k:=A/\m)$ of dimension $d,$ we denote by $h$ the
 embedding codimension of $A$, namely the integer $h:=\dim_k(\m/\m^2)-d.$ Recall, see \cite{Ab},
 that the multiplicity $e$ of  a Cohen-Macaulay local ring $A$ of embedding codimension $h$
 satisfies the inequality  $e\ge h+1.$ Further, in the extremal case $e=h+1,$ it is well known that
  $\mathbb{P}_A(z)$ is rational.

 The main result of this paper is the following theorem.

  \begin{theorem} Let $A=R/I$ be an almost stretched Gorenstein local ring of dimension
  $d$ and embedding codimension $h.$ Then $$\mathbb{P}_A(z)=\frac{(1+z)^d}{1-hz+z^2}.$$
 \end{theorem}

 \bigskip
 \section{Proof of the Theorem}

 The main ingredient of the proof of our result are the following classical
 ``change of rings" theorems.
 The first one, see \cite{T}, relates the Betti numbers of $A$ with those of
 $A/xA$ when $x$ is a non-zero divisor in the local ring $A.$

\vskip 2mm
a) Let $x$ be a non-zero divisor in $A.$ Then

 $$\mathbb{P}_A(z)=\begin{cases}
      (1+z)\mathbb{P}_{A/xA}(z) &  \ \ x\in \m\setminus \m^2\\
   (1-z^2)\mathbb{P}_{A/xA}(z)   & \ \ x\in \m^2 .
\end{cases}$$

\vskip 2mm The second one, see \cite{GL}, relates the Betti numbers of $A$ with
those of $A/xA$ when $x$ is a socle element.
\vskip 2mm
b) Let $x\in \m\setminus \m^2$ be an element in the socle $(0:_A\m)$ of $A.$
Then $$\mathbb{P}_A(z)=\frac{\mathbb{P}_{A/xA}(z)}{1-z \ \mathbb{P}_{A/xA}(z)}.$$

\vskip 2mm The third one, see \cite{AL}, relates the Betti numbers of the
Artinian Gorenstein local ring $A$ with those of $A$ modulo the socle.
\vskip 2mm
c)  If $(A,\m)$ is an Artinian local Gorenstein ring, then
$$\mathbb{P}_A(z)=\frac{\mathbb{P}_{A/(0:\m)}(z)}{1+z^2 \ \mathbb{P}_{A/(0:\m)}(z)}.$$

We start now proving the Theorem. Let $J:=(a_1,\dots,a_d)$ be the ideal generated
by a minimal reduction of $\m,$ such that $A/J$ is almost stretched and Gorenstein.
 Since $\{a_1,\dots,a_d\}$ is a regular sequence on $A,$ we have by a)
 $$\mathbb{P}_A(z)=(1+z)^d \ \mathbb{P}_{A/J}(z).$$

Hence we may assume that $(A=R/I,\m=\n/I)$ is an Artinian almost stretched
Gorenstein local ring of embedding dimension  $h.$
In this case we   proved in  \cite{EV1}, Proposition 4.8,  that we can find
integers $s\ge t+1\ge 3$ depending on the Hilbert function of $A$,   a minimal
system of generators $\{x_1,\dots,x_h\}$ of the maximal ideal $\n$ of $R$ and an
element $a\in R$ such that $I$ is generated by the elements:
$$ \{x_1x_j\}_{j=3,\dots,h}\ \ \{x_ix_j\}_{2\le i<j\le h}\ \ \{x_j^2-x_1^s\}_{j=3,\dots,h}\ \
x_2^2-ax_1x_2-x_1^{s-t+1},\ \ x_1^tx_2.$$ Further $\overline{x_1}^s\in A=R/I$
is the generator of the socle $0:\m$ of $A.$ Hence $$A/(0:\m)\simeq
R/(I+(x_1^s))=R/K$$ where $K$ is the ideal in $R$ generated by
$$ \{x_1x_j\}_{j=3,\dots,h}\ \ \{x_ix_j\}_{2\le i<j\le h}\ \ \{x_j^2\}_{j=3,\dots,h}\ \
x_2^2-ax_1x_2-x_1^{s-t+1},\ \ x_1^tx_2,\ \ x_1^s.$$
Notice that by c) we have
\begin{equation}\label{A}\mathbb{P}_A(z)=\frac{\mathbb{P}_{A/(0:\m})(z)}{1+z^2 \
 \mathbb{P}_{A/(0:\m)}(z)}=\frac{\mathbb{P}_{R/K}(z)}{1+z^2 \ \mathbb{P}_{R/K}(z)}
 \end{equation}
 so that we are left to compute the Poincar{\'e} series of $R/K.$

It is clear that $\overline{x}_3,\dots,\overline{x}_h\in \m  \setminus \m^2 $ are
elements in the socle of $R/K.$ Hence we can use $h-2$ times b)  to get
\begin{equation}\label{K}\mathbb{P}_{R/K}(z)=\frac{\mathbb{P}_{S/L}(z)}{1-(h-2)z\ \mathbb{P}_{S/L}(z)}
\end{equation} where $S=R/(x_3,\dots,x_h)$ is a two dimensional regular local ring with maximal
ideal $\n=(x_1,x_2)$ and $L$ is the ideal
$$L:=(x_2^2-ax_1x_2-x_1^{s-t+1},\ \ x_1^tx_2,\ \ x_1^s).  $$
Now let $V:=(x_2^2-ax_1x_2-x_1^{s-t+1},\ \ x_1^tx_2);$  in \cite{EV2},
Theorem 4.7 we proved that $S/V$ is an almost stretched Artinian
Gorenstein local ring with socle generated by $\overline{x}_1^s.$
Hence, by using c), we get
\begin{equation}\label{V} \mathbb{P}_{S/V}(z)=\frac{\mathbb{P}_{S/L}(z)}{1+z^2\ \mathbb{P}_{S/L}(z)}.
\end{equation}
Finally, the  ideal $V$ is generated by a regular sequence in $\n^2$ so that, by a),
we get $$\mathbb{P}_{S/V}(z)=\frac{(1+z)^2}{(1-z^2)^2}=\frac{1}{(1-z)^2}.$$
By using (\ref{V}), this last equality gives  $\mathbb{P}_{S/L}(z)=\frac{1}{1-2z}$
and thus  from (\ref{K}) we get $\mathbb{P}_{R/K}(z)=\frac{1}{1-hz}.$
Finally, by (\ref{A}) we get $$\mathbb{P}_A(z)=\frac{1}{1-hz+z^2}$$ and the conclusion follows.
\vskip 2mm
 A consequence of the above Theorem is the rationality of the Poincar{\'e} series of
 any Gorenstein local ring of multiplicity $e\le h+4.$
\begin{corollary} Let $A$ be a Gorenstein local ring of dimension $d$, multiplicity $e$
and embedding codimension $h$. If $e=h+1,h+2,h+3,h+4$ then $\mathbb{P}_A(z)$ is rational.
\end{corollary}\begin{proof} We need only to consider the case $h+2\le e\le h+4.$
Since any Artinian reduction $B$ of $A$ is a Gorenstein local ring with  the same
multiplicity and the same embedding codimension, the possible Hilbert series of
$B$ are $$\{1,h,1\},\{1,h,1,1\},\{1,h,2,1\},\{1,h,1,1,1\}.$$
This proves that $A$ is either stretched or almost stretched.
The conclusion follows by Sally's result and the above Theorem.
\end {proof}

\begin{corollary} If $A$  is a Gorenstein local ring of multiplicity
at most seven, then $\mathbb{P}_A(z)$ is rational.
\end{corollary}
\begin{proof} As before the possible Hilbert series of any artinian reduction
$B$ of $A$ are $$\{1,5,1\},\{1,4,1,1\},\{1,3,2,1\},\{1,3,1,1,1\},$$
$$\{1,2,3,1\},\{1,2,2,1,1\},\{1,2,1,1,1,1\},\{1,1,1,1,1,1,1\}.$$
But $\{1,2,3,1\}$ is not allowed because Gorenstein in codimension two implies complete
intersection. In all the remaining cases $A$ is either stretched or almost stretched and we get the conclusion.
 \end{proof}

\vskip 2mm
 \noindent
 {\bf Remark 1.}
B{\o}gvad in  \cite{Bog83} showed that there exist  Artinian
Gorenstein local rings of multiplicity $26$ with non-rational
Poincar\'{e} series.

 \vskip 2mm
 \noindent
 {\bf Remark 2.} The Hilbert function of an  almost stretched Artinian
 Gorenstein local ring $A$ has the following shape
 \begin{center} \begin{tabular}{|c|c|c|c|c|c|c|c|c|c|c|c|c|}
n &0 & 1  & 2 & \dots & t & t+1  &\dots & s  & s+1  \\ \hline
$H_A(n)$ &1  & h  & 2 & \dots & 2 & 1 &  \dots & 1  &  0  \\
\end{tabular} \end{center}  for integers $s$ and $t$ such that $s\ge t+1 \ge 3.$
On the contrary the Poincar{\'e} series of $A$ is independent from $t$ and $s.$

\vskip 2mm \noindent {\bf Remark 3.} The Poincar{\'e} series of
stretched and almost stretched Gorenstein local rings with the same
dimension and the same embedding dimension coincide, see \cite{S2}
Theorem 2.

\providecommand{\bysame}{\leavevmode\hbox to3em{\hrulefill}\thinspace}

\bigskip
\bigskip
\noindent
Juan Elias\\
Departament D'{\`A}lgebra i Geometria \\
Facultat de Matem\`{a}tiques\\
Universitat de Barcelona\\
Gran Via 585, 08007 Barcelona, Spain
e-mail: {\tt elias@ub.edu}

\bigskip
\noindent
Giuseppe Valla\\
Dipartimento di Matematica\\
Universit{\`a} di Genova\\
Via Dodecaneso 35, 16146 Genova, Italy\\
e-mail: {\tt valla@dima.unige.it}

\end{document}